\documentclass[12pt]{amsart}
\usepackage{amsfonts}
\usepackage{amscd}
\usepackage{amssymb}
\usepackage{graphics}

\usepackage{amsmath}

\usepackage{hyperref}
\hypersetup{colorlinks,citecolor=blue}

\newtheorem{prop}{Proposition}[section]

\newcommand{\BZ}{{\mathbb{Z}}}
\newcommand{\BN}{{\mathbb{N}}}
\newcommand{\BR}{{\mathbb{R}}}

\newcommand{\BQ}{{\mathbb{Q}}}
\newcommand{\BP}{{\mathbb{P}}}

\newcommand{\BG}{{\mathbb{G}}}

\newcommand{\OO}{{\mathcal{O}}}

\newcommand{\gD}{\Delta}
\newcommand{\gd}{\delta}

\newcommand{\gC}{\Gamma}
\newcommand{\gc}{\gamma}

\newcommand{\gO}{\Omega}

\newcommand{\gep}{\epsilon}

\newcommand{\p}{\prod}

\newcommand{\SL}{\text{SL}}
\newcommand{\GL}{\text{GL}}

\newcommand{\SO}{\text{SO}}

\newtheorem{thm}[prop]{Theorem}
\newtheorem{lem}[prop]{Lemma}

\theoremstyle{definition}

\newtheorem{rem}[prop]{Remark}

\newtheorem{asmp}[prop]{Assumption}

\def\Z{\Bbb Z}

\def\P{\Bbb P}

\def\R{\Bbb R}

\def\P{\Bbb P}

\def\p^k{\tilde{\P}^k}

\catcode`\@=11
\long\def\@savemarbox#1#2{\global\setbox#1\vtop{\hsize\marginparwidth 
  \@parboxrestore\tiny\raggedright #2}}
\catcode`\@=12

\begin{document}
\author{T. Gelander and C. Meiri} 

\address{Mathematics and Computer Science\\
Weizmann Institute\\
Rechovot 76100, Israel\\}
\email{tsachik.gelander@gmail.com}

\address{Department of Mathematics\\ Technion - Israel Institute of Technology\\
Haifa, 32000, Israel\\}
\email{chenm@technion.ac.il}

\date{\today}


\title{Higher rank arithmetic groups which are not invariably generated}

\begin{abstract}
It was conjectured in \cite{KLS} that for arithmetic groups, Invariable Generation is equivalent to the Congruence Subgroup Property. In view of the famous Serre conjecture this would imply that higher rank arithmetic groups are invariably generated. 
In this paper we prove that some higher rank arithmetic groups are not invariably generated.
\end{abstract}

\maketitle

A subset $S$ of a group $\gC$ {\it invariably generates} $\gC$ if $\gC= \langle s^{g(s)} | s \in S\rangle$ for every choice of $g(s) \in \gC,s \in S$. One says that a group $\gC$ is {\it invariably generated}, or shortly IG if one of the following equivalent conditions holds:
\begin{enumerate}
\item There exists $S \subseteq \gC$ which invariably generates $\gC$.
\item  The set $\gC$ invariably generates $\gC$.
\item Every transitive permutation representation on a non-singleton set admits a fixed-point-free element.
\item $\gC$ does not have a proper subgroup which meets every conjugacy class.
\end{enumerate}

In an attempt to give a purely algebraic characterization for the Congruence Subgroup Property  (CSP), it was asked in \cite{KLS} whether for arithmetic groups in simply connected, connected, simple Lie groups the CSP is equivalent to IG.  This was partly motivated by the result of \cite{KLS} that the pro-finite completion of an arithmetic group with the CSP is topologically finitely invariably generated. Another result in the positive direction was obtained in \cite{Gel}, namely that discrete subgroups, and in particular arithmetic subgroups, of rank one simple Lie groups are not IG. Recall the famous Serre conjecture that an irreducible arithmetic lattice $\gC$ in a simply connected, connected, simple Lie group $G$ has the CSP iff $\text{rank}_\BR(G)\ge 2$. In this paper we give examples of arithmetic lattices in $\SL_3(\R)$ and other higher rank groups which are not IG. If Serre's conjecture is true,  these examples give a negative answer the question of \cite{KLS11}.

Let $L$ be a cubic Galois extension of $\BQ$. Let $D$ be a central division algebra of degree $3$ over $\BQ$ which contains $L$ as a subfield. Let $\OO_D$ be an order in $D$. Then $\Gamma_D:=\SL_1(\OO_D)$ is isomorphic to an anisotropic arithmetic lattice in $\SL(3,\BR)$ (See \cite[Proposition 6.61]{Witte}). 
The corresponding lattice is co-compact. By Serre's conjecture, $\Gamma_D$ is expected to possess the CSP, but this is currently unknown.

\begin{thm}\label{thm2} 
The arithmetic group $\Gamma_D$ is not IG.
\end{thm}

Theorem \ref{thm2} is a special case of the more general Theorem \ref{thm:regular}.

\medskip



Suppose that $G=\BG(\BR)$ is a Zariski connected linear real algebraic group with a faithful irreducible representation in a $d$-dimensional real vector space, with associated projective 
space $\BP\cong\BP^{d-1}(\BR)$.
We identify $G$ with its image in $\GL_d(\BR)$. 
An element $g\in G$ is {\it regular} if it is either a regular unipotent or all its eigenvalues are of distinct absolute value (in particular each has multiplicity one). A subgroup $\gC\le G$ is regular if every nontrivial element of $\gC$ is regular.
For an element $g\in G$ we denote by $v_g$ and $H_g$ the attracting point and repelling hyperplane of $g$ in $\BP$, assuming they are unique. 
Let us say that a point or a hyperplane is $G$-{\it defined} if it is of the form $v_g$ (resp. $H_g$) for some $g\in G$. In addition let us say that an incident pair $(v,H)$ is $G$-defined if there is a regular element $g\in G$ for which $v=v_g$ and $H=H_{g^{-1}}$. Let us also say that an ordered pair $(U,V)$ of open neighborhoods in $\BP$ is $G$-{\it applicable} if there is a $G$-defined incident pair $(v,H)$ with $v\in U$ and $H$ entirely contained in the interior of the complement of $V$. 
We shall further suppose the following:

\begin{asmp}\label{asmp:mani-pairs}
The set of $G$-defined incident pairs is not contained in a proper sub-variety of the set of incident pairs.
\end{asmp}


\begin{thm}\label{thm:regular}
Let $G$ be a Zariski connected linear real algebraic group with a faithful irreducible representation on a $d$-dimensional vector space
which satisfies Assumption \ref{asmp:mani-pairs} and let $\gC\le G$ be a regular lattice. 
Then $\gC$ admits an infinite rank free subgroup which intersects every conjugacy class. In particular, $\gC$ is not invariably generated.
\end{thm}

\begin{rem}
For $G=\SL_d(\BR)$ with its standard $d$-dimensional representation Assumption \ref{asmp:mani-pairs} is obviously satisfied. Moreover it is easy to check that the lattices in Theorem  \ref{thm2} are regular, by construction. Thus Theorem \ref{thm2} is a special case of Theorem \ref{thm:regular}. 
\end{rem}




\begin{prop}[Main Proposition]\label{prop:main}
Let $C$ be a nontrivial conjugacy class in $\gC$, and let $(U,V)$ be a $G$-applicable pair of open sets in $\BP$. Then there is $\gc\in C$ such that $\gc^n\cdot V,~n\in\BZ\setminus\{ 0\}$ are all disjoint and contained in $U$.
\end{prop}


\begin{lem}\label{lem:the-conjugating}
Let $\gc\in \gC\setminus\{1\}$, and let $(U,V)$ be a $G$-applicable pair of open sets in $\BP$.  Then there is $\gd\in\gC$ such that 
\begin{enumerate}
\item $H_{\gd^{-1}}\cap \overline{V}=\emptyset$,
\item $v_\gc,v_{\gc^{-1}}\notin H_\gd$, 
\item $v_{\gd^{-1}}\notin H_\gc\cup H_{\gc^{-1}}$,
\item $v_\gd\in U$, and
\item $\gc^n\cdot v_{\gd^{-1}}\notin H_\gd$ for all $n\ne 0$.
\end{enumerate}
\end{lem}

\begin{proof}
As a first step, observe that there exists a regular $g\in G$ such that all the five conditions hold with $g$ instead of $\gd$. Indeed, as $(U,V)$ is $G$-applicable there is $g$ for which $(1)$ and $(4)$ are satisfied. 
Moreover, in view of Assumption \ref{asmp:mani-pairs}, for every $n\in\BN\setminus\{0\}$ the set of regular $g\in G$ such that 
$\gc^n\cdot v_{g^{-1}}\in H_g$ is nowhere-dense. Hence by the Baire category theorem and the fact that $G$ admits an analytic structure, we may slightly deform $g$ in order to guarantee that condition $(5)$ holds as well.
Since $G$ is Zariski connected and irreducible, the set of $h\in G$ such that either $(2)$ or $(3)$ do not hold for $g^h$ is nowhere-dense. Thus, up to replacing $g$ by a conjugate $g^h$ with $h$ sufficiently close to $1$, all the five conditions are satisfied.

Next note that $\gc^n\cdot v_{g^{-1}}$ tends to $v_\gc$ when $n\to\infty$ and to $v_{\gc^{-1}}$ when $n\to-\infty$. Thus if we pick a sufficiently small symmetric identity neighborhood $\gO\subset G$ and $\gep>0$ sufficiently small, then we have:
\begin{enumerate}
\item $\gO\cdot (H_{g^{-1}})_\gep\cap \overline{V}=\emptyset$,
\item $v_\gc,v_{\gc^{-1}}\notin \gO\cdot (H_g)_\gep$, 
\item $\gO\cdot (v_{g^{-1}})_\gep\cap (H_\gc\cup H_{\gc^{-1}})=\emptyset$,
\item $\gO\cdot (v_g)_\gep\subset U$, 
\item $\gc^n\gO\cdot (v_{g^{-1}})_\gep\cap \gO\cdot (H_g)_\gep=\emptyset$ for all $n\ne 0$, and
\item $\gO\cdot (v_g)_\gep\cap\gO\cdot (H_g)_\gep=\emptyset=\gO\cdot (v_{g^{-1}})_\gep\cap\gO\cdot (H_{g^{-1}})_\gep$.
\end{enumerate}

Finally we wish to replace $g$ by an element $\gd\in \gC$ with similar dynamical properties. 
Since $\gC$ is a lattice, $G/\gC$ carries a $G$-invariant probability measure. Let $\pi:G\to G/\gC$ denote the quotient map. By the Poincar\'{e}  recurrence theorem we have that for arbitrarily large $m$, $g^m\cdot\pi(\gO)\cap\pi(\gO)\ne\emptyset$, which translates to $\gO g^m\gO\cap\gC\ne\emptyset$, since $\gO$ was chosen to be symmetric. 
 
Note that replacing $g$ by a power $g^m,~m>0$ does not change the attracting points and repelling hyperplanes, while for larger and larger $m$ the element $g^m$ is becoming more and more proximal. In particular, if $m$ is sufficiently large then 
$$
 g^m\cdot (\BP\setminus (H_g)_\gep)\subset (v_g)_\gep,~\text{and}~g^{-m}\cdot (\BP\setminus (H_{g^{-1}})_\gep)\subset (v_{g^{-1}})_\gep.
$$
In view of Item (6), this implies that any element of the form $h=w_1g^mw_2$ with $w_1,w_2\in\gO$ is very proximal and satisfies 
$$
 v_h\in\gO\cdot(v_g)_\gep,~v_{h^{-1}}\in\gO\cdot(v_{g^{-1}})_\gep,~H_h\subset\gO\cdot (H_g)_\gep,~H_{h^{-1}}\subset\gO\cdot (H_{g^{-1}})_\gep.
$$
In particular, we obtain the desired $\gd$ by choosing any element from the nonempty set $\gC\cap\gO g^m\gO$.
\end{proof}
%
%

\begin{proof}[Proof of Proposition \ref{prop:main}]
Pick $\gc'$ in $C$ and let $\gd$ be the element provided by Lemma \ref{lem:the-conjugating}. Items (2,3,5) of the lemma together with the contraction property of high powers of $\gc'$ guarantee that if $O$ is a sufficiently small neighborhood $v_{\gd^{-1}}$ then its $\langle\gc'\rangle$ orbit stays away from $H_\gd$, i.e. there is a neighborhood $W$ of $H_\gd$ disjoint from all $\gc'^n\cdot O,~n\in\BZ\setminus\{0\}$. In particular the $\langle\gc'\rangle$ translates of $O$ are all disjoint, i.e. $O$ is $\langle\gc'\rangle$-wondering (in the terminology of \cite{Gel}). If $k$ is sufficiently large then by Item (1) of the lemma $\gd^{-k}\cdot V\subset O$ while by Item (4), $\gd^k\cdot (\BP\setminus W)\subset U$. It follows that $\gc=\gd^k\gc'\gd^{-k}$ is our desired element.

\end{proof}

We shall also require the following:

\begin{lem}\label{lem:(U,V)-pairs}
There is a countable family of disjoint open sets $U_k,~k\in\BN$ such that if we let $V_k=\bigcup_{i\ne k}U_i$ for any $k\in\BN$ then the pairs $(U_k,V_k),~k\in\BN$
are all $G$-applicable.
\end{lem}

\begin{proof}
Let $g\in G$ be a regular element and set $(v,H)=(v_g,H_g)$. Pick $h\in G$ near the identity such that $h\cdot v\notin H$ and $v\notin h\cdot H$. Let $c:[0,1]\to G$ be an analytic curve with $c(0)=g$ and $c(1)=h$. Since the curve $c$ is analytic, the boundary conditions allow us to chose a sequence $t_k\to 1$ such that if we let $(v_k,H_k)=(c(t_k)\cdot v,c(t_k)\cdot H)$ and $(v_\infty,H_\infty)=(h\cdot v,h\cdot H)$, then
for all $i\ne j\le\infty$ we have $v_i\notin H_j$. In addition, the set of pairs 
$\{(v_k,H_k):k<\infty\}$ is discrete and accumulates to $(v_\infty,H_\infty)$. Thus, we may choose neighborhoods $U_k\supset v_k$ sufficiently small so that $\overline{U_i}\cap H_j=\emptyset,~\forall i\ne j$ and so that the only point in $\overline{\cup U_k}\setminus\cup\overline{U_k}$ is $h\cdot v$. It is evident that the sets $U_k,~k\in\BN$ satisfy the required property.
\end{proof}

The strategy of the proof of Theorem \ref{thm:regular} is as follows.  Let $(U_k,V_k),~k\in\BN$ be as in Lemma \ref{lem:(U,V)-pairs}.
Let $C_k$ be an enumeration of the non-trivial conjugacy classes in $\gC$. In view of Proposition \ref{prop:main}, we may pick $\gc_k\in C_k$ so that 
$$
 \gc_k^n\cdot V_k\subset U_k,~\forall n\in\BZ\setminus \{0\}.
$$

We derive Theorem \ref{thm:regular} from  
the following variant of the ping-pong lemma (cf.\ \cite[Proposition 3.4]{Gel}:

\begin{prop}
Let $\gC$ be a group and $X$ a compact $\gC$-space. Let $\gc_k\in \gC,~k\in\BN$ be elements of $\gC$, and suppose there are disjoint subsets $U_k\subset X$ such that for all $k\ne j$ and all $n\in\BZ\setminus\{0\}$ we have $\gc_k^n\cdot U_j\subset U_k$. Then 
$$
 \gD=\langle \gc_k,~k=1,2,3,\ldots\rangle=*_k\langle\gc_k\rangle,
$$
is the free product of the infinite cyclic groups $\langle\gc_k\rangle,k=1,2,3,\ldots$.

Furthermore, the limit set $L(\gD)$ is contained in $\overline{\cup_k U_k}$.
\end{prop} 
\qed

%
%
%
%

\begin{rem}
An earlier version of this paper \cite{paper} was called ``The congruence subgroup property  does not imply invariable generation''.	
Theorem 1.1 in the earlier version stated that if $q$ is a rational quadratic from of signature $(2,2)$, then $\SO_q(\Z)$ is not IG. However, the proof of that theorem in \cite{paper} contained a gap which we do not know how to fix.
\end{rem}



\begin{thebibliography}{999999}


\bibitem[D92]{Dixson}
J. D. Dixon, Random sets which invariably generate the symmetric group. Discrete Math. 105 (1992), 25--39.

\bibitem[HNN49]{HNN}
G. Higman, B. H. Neumann and H. Neumann, Embedding theorems for groups, J. London
Math. Soc. 24 (1949), 247--254.

\bibitem[G15]{Gel}
T. Gelander, Convergence groups are not invariably generated, to appear in IMRN.

\bibitem[GM17]{paper} T. Gelander, C. Meiri, The congruence subgroup property does not imply invariable generation. Int. Math. Res. Not. IMRN 2017, no. 15, 4625--4638.
 


\bibitem[KLS11]{KLS11}
W. M. Kantor, A. Lubotzky and A. Shalev, Invariable generation and the Chebotarev
invariant of a finite group, J. Algebra 348 (2011), 302--314.

\bibitem[KLS14]{KLS} W.M. Kantor, A. Lubotzky, A. Shalev, Invariable generation on infinite groups,  J. Algebra 421 (2015), 296--310.


\bibitem[KM79]{Kneser} M. Kneser, Normalteiler ganzzahliger Spingruppen, J. Reine und Angew. Math. 311/312(1979), 191--214.


\bibitem[R72]{Rag} M.S. Raghunathan, {\it Discrete Subgroups of Lie Groups}, 
Springer, New York, 1972.

\bibitem[R76]{Ra1} M.S. Raghunathan, On the congruence subgroup problem. Inst. Hautes Etudes Sci. Publ. Math. No. 46 (1976), 107--161.

\bibitem[R86]{Ra2} M.S. Raghunathan, On the congruence subgroup problem. II. Invent. Math. 85 (1986), no. 1, 73--117.

\bibitem[S14]{Sa}  P. Sarnak, Notes on thin matrix groups. Thin groups and superstrong approximation, Math. Sci. Res. Inst. Publ., 61, Cambridge Univ. Press, Cambridge, (2014), 343--362.

\bibitem[W76]{W76} J. Wiegold, Transitive groups with fixed-point-free permutations, Arch. Math. (Basel) 27
(1976), 473--475.

\bibitem[We77]{W77} J. Wiegold, Transitive groups with fixed-point-free permutations. II, Arch. Math. (Basel)
29 (1977), 571--573.

\bibitem[WM14]{Witte} D. Witte--Morris, Introduction to Arithmetic Groups.
\end{thebibliography}
\end{document}